\documentclass{amsart}
\usepackage{amsfonts,amssymb,latexsym,amsthm,amsmath}
\usepackage{mathrsfs}
\newtheorem{theorem}{Theorem}[section]
\newtheorem{example}[theorem]{Example}
\newtheorem{lemma}[theorem]{Lemma}
\newtheorem{remark}[theorem]{Remark}

\newtheorem{cor}[theorem]{Corollary}

\DeclareMathAlphabet\mathbb{U}{msb}{m}{n}

\usepackage{tikz}

\usetikzlibrary{matrix}
\usepackage[arrow,matrix,curve]{xy}\SilentMatrices
\def\xyma{\xymatrix@M.7em}

\begin{document}
\title{Leibniz rule on higher pages of unstable spectral sequences}
\author{Sergei O. Ivanov}
\address{Chebyshev Laboratory, St. Petersburg State University, 14th Line, 29b,
Saint Petersburg, 199178 Russia} \email{ivanov.s.o.1986@gmail.com}

\author{Roman Mikhailov}
\address{Chebyshev Laboratory, St. Petersburg State University, 14th Line, 29b,
Saint Petersburg, 199178 Russia and St. Petersburg Department of
Steklov Mathematical Institute} \email{rmikhailov@mail.ru}
\author{Jie Wu }
\address{Department of Mathematics, National University of Singapore, 10 Lower Kent Ridge Road, Singapore 119076} \email{matwuj@nus.edu.sg}
\urladdr{www.math.nus.edu.sg/\~{}matwujie}

\thanks{The research is supported by Russian Scientific Foundation, grant N 14-21-00035.}

\begin{abstract}
A natural composition $\odot$ on all pages of the lower central
series spectral sequence for spheres is defined. Moreover, it is
defined for $p$-lower central series spectral sequence of a
simplicial group. It is proved that $r$th differential satisfies a
``Leibiz rule with suspension'': $d^r(a\odot \sigma b)=\pm
d^ra\odot b+a\odot d^r\sigma b,$ where $\sigma$ is the suspension
homomorphism.
\end{abstract}

 \maketitle

\section{Introduction}

It often happens that, different spectral sequences, which play a
fundamental role in algebraic topology, have additional structures
like smash and composition pairing, Whitehead products etc. In
particular, the main homotopy spectral sequences such as
Bousfield-Kan and Adams spectral sequences have a composition
pairing \cite{Bousfield-Kan2}. In the stable range the composition
for the homotopy spectral sequences are introduced naturally. For
the unstable spectral sequences the situation is different.

In a practice, any additional structure on a spectral sequence gives a lot of information on its differentials.
In this paper, the authors discuss mainly the spectral sequences arising from filtrations and consider general
properties needed for the definition of composition in the unstable case.

The main motivating example is the following.
For a simplicial group $G$, and an odd prime $p$, let $E(G,
\mathbb F_p)$ be the $p$-lower central series spectral sequence
(see \cite{Rector}, \cite{Bousfield_Curtis}, \cite{6_authors}). For a simplicial set $X$ we set $E(X,\mathbb F_p)=E(F[X],\mathbb F_p),$ where $F[X]$ is Milnor's construction. It is well-known that $E(X,\mathbb F_p) \Rightarrow \pi_*(\Sigma X,p),$ for 'good' spaces $X$, where $\pi_*(\Sigma X,p)$ is the quotient of $\pi_*(\Sigma X)$ by prime to $p$ torsion \cite{Rector} (we use  $F[X]=G\Sigma X,$ where $G$ is the Kan's loop functor). Moreover, $E(S^{2n},\mathbb F_p)$ is described on the language of the lambda algebra \cite{6_authors}. Denote $E(S^*,\mathbb F_p)=\bigoplus_{i\geq 2} E(S^i,
\mathbb F_p).$ We show that (Theorem
\ref{simplicial_theorem}) for $r\geq 1,$ there exists a
well-defined map
$$\odot: E^r(G,\mathbb F_p)\times E^r(S^*,\mathbb
F_p)\longrightarrow E^r(G,\mathbb F_p),$$ such that for elements
$a\in E^r_{k,l}(G,\mathbb F_p)$ and $b\in E^r_{s,t}(S^*,\mathbb
F_p)$ the following Leibniz-like rule holds
$$
d^r(a\odot \sigma b)=(-1)^{l+t} d^r a \odot b+a\odot d^r \sigma b,
$$
where $\sigma$ is the suspension homomorphism. Moreover, this composition coincides with the multiplication on the lambda algebra.

This composition operation is a consequence of the more general result,
\hbox{Theorem \ref{Theorem_E(S^*,M)}}, which is formulated in a general
setting, for a filtered monad over Milnor's construction on a pointed simplicial set and the Kleisli composition. Such a
general approach allows to introduce a natural composition on
different unstable spectral sequences defined via filtrations of
simplicial groups.

\section{Kleisli composition and suspension homomorphism}

\subsection{Kleisli composition}

Let $(\mathcal M,{\sf m},{\sf e})$ be a monad on the category of
pointed sets $\mathcal M:{\sf Sets}_*\to {\sf Sets}_*,$ where
\hbox{${\sf m}:\mathcal M^2 \to \mathcal M$} is the multiplication
morphism and ${\sf e}:{\sf Id}\to \mathcal M$ is the identity
morphism. Let $A$ be a $\mathcal M$-algebra in ${\sf Sets}_*$ with
the structure morphism ${\sf m}^A:\mathcal MA\to A$ and $X,Y$ be a
two pointed sets. Then we define the {\bf Kleisli composition}
(see VI.5 of \cite{Mac_Lane}) as follows
$${\sf Sets}_*(X,A)\times {\sf Sets}_*(Y, \mathcal M X) \longrightarrow {\sf Sets}_*(Y,A)$$
$$(f,g)\longmapsto {\sf m}^A \circ \mathcal M f\circ g  =:f\odot g.$$
This map is natural by $Y$ and $A.$ In other words, if $h:Y'\to Y$
is a point-preserving map and $\alpha:A\to A'$ is a homomorphism
of $\mathcal M$-algebras, then
\begin{equation}\label{eq_M-comp1}
(f\odot g) \circ h=f\odot (g\circ h)
\end{equation}
and
\begin{equation}\label{eq_M-comp2}
\alpha\circ (f\odot g)=(\alpha\circ f)\odot g.
\end{equation}
The first equality is obvious. The second holds because
$\alpha\circ (f\odot g)=\alpha\circ {\sf m}^A \circ \mathcal M
f\circ g={\sf m}^B\circ \mathcal M\alpha \circ \mathcal Mf \circ
g={\sf m}^B\circ \mathcal M(\alpha\circ f) \circ g=(\alpha\circ f)
\odot g.$ Hence we can write compositions $\alpha\circ f\odot
g\circ h$ without brackets.

Extend the monad $(\mathcal M,{\sf m},{\sf e})$ on the category of
pointed simplicial sets $\mathcal S_*:$
$$\mathcal M:\mathcal S_*\longrightarrow \mathcal S_*,$$
and extend the Kleisli composition
$$\odot: \mathcal S_*(X,A)\times \mathcal S_*(Y,\mathcal MX)\longrightarrow \mathcal S_*(Y,A),$$ by the same formula, where $X,Y\in \mathcal S_*$ and $A$ is a simplicial $\mathcal M$-algebra. It is easy to see that the properties \eqref{eq_M-comp1} and \eqref{eq_M-comp2} still holds.

Consider the case of $Y=\Delta[m]^+=\Delta[m]\sqcup *$. Since $\mathcal S_*(\Delta[m]^+,A)=A_m$ and $\mathcal
S_*(\Delta[m]^+,\mathcal MX)=\mathcal MX_m$ we get a map
\begin{equation}\label{eq_M_action}
\odot : \mathcal S_*(X,A)\times \mathcal MX_m\longrightarrow A_m.
\end{equation}
\begin{lemma} Let $A$ be a simplicial $\mathcal M$-algebra and $f\in\mathcal S_*(X,A).$
Then $$f\odot -:\mathcal MX\to A$$ is the unique morphism  of simplicial $\mathcal M$-algebras such that $(f\odot -)\circ {\sf e}_X=f$.
In particular, $
d_i(f\odot g)=f\odot d_ig$ and $s_i(f\odot g)=f\odot s_ig$ for all
$i$ and $g\in \mathcal MX.$
\end{lemma}
\begin{proof}
First we prove that $f\odot -$ is a simplicial map. Indeed,
consider the coface map $d^i:\Delta[n]\to \Delta[n+1].$  Applying
\eqref{eq_M-comp1}, we get $$d_i(f\odot g)=(f\odot g) \circ
(d^{i})^+=f\odot (g\circ (d^i)^+)=f\odot d_ig.$$ Similarly we obtain
$s_i(f\odot g)=f\odot s_ig.$

Prove that $f\odot -$ is a morphism of $\mathcal M$-algebras. We
have to prove the equality $${\sf m}^A\circ \mathcal M(f\odot -)=
(f\odot -) \circ {\sf m}_{X}.$$ Note that $f\odot -={\sf m}^A\circ
\mathcal Mf.$ Then \begin{multline*} {\sf m}^A\circ \mathcal
M(f\odot -)= {\sf m}^A\circ \mathcal M{\sf m}^A\circ \mathcal
M^2f=\\ {\sf m}^A\circ {\sf m}_{A}\circ \mathcal M^2f={\sf
m}^A\circ \mathcal Mf\circ {\sf m}_{X}=(f\odot -) \circ {\sf
m}_{X}.\end{multline*} Thus $f\odot -$ is a morphism of $\mathcal
M$-algebras.

Since $\mathcal MX$ is a free $\mathcal M$-algebra generated by
$X,$ any morphism of $\mathcal M$-algebras $\mathcal MX\to A$ is
uniquely defined by the composition $X\overset{{\sf e}_X}\to
\mathcal MX\to A.$  Since $$(f\odot -)\circ {\sf e}_X=f\odot {\sf
e}_X={\sf m}^A\circ \mathcal Mf\circ {\sf e}_X={\sf m}^A\circ {\sf
e}_A\circ f=f,$$ we obtain that $f\odot -$ is the unique morphism
of simplicial algebras such that $(f\odot -)\circ {\sf e}_X=f.$
\end{proof} Denote by ${\bf i}_n$ the non-degenerate $n$-simplex in $\Delta[n]$ and its image in $S^n.$
If we take $X=\Delta[n]^+$ and $X=S^n,$ and use that $\mathcal
S_*(\Delta[n]^+,X)\cong X_n$ and $\mathcal S_*(S^n,X)\cong
Z_nX=\{x\in X_n\mid d_ix=* \text{ for } 0\leq i\leq n\}$  we get
the following corollaries.
\begin{cor}
Let $A$ be a simplicial $\mathcal M$-algebra and $a\in A_n.$ Then
$$a\odot -:\mathcal M(\Delta[n]^+)\to A$$ is the unique morphism  of
simplicial $\mathcal M$-algebras  that sends ${\sf e}{\bf i}_n$ to
$a.$ In particular, $ d_i(a\odot x)=a\odot d_ix$ and $s_i(a\odot
x)=a\odot s_ix$ for all $i$ and $x\in \mathcal M(\Delta[n]).$
\end{cor}
\begin{cor}\label{Corolary_Sph}
Let $A$ be a simplicial $\mathcal M$-algebra and $a\in Z_nA.$ Then
$$a\odot -:\mathcal M(S^n)\to A$$ is the unique morphism  of
simplicial $\mathcal M$-algebras  that sends ${\sf e}{\bf i}_n$ to
$a.$ In particular, $ d_i(a\odot x)=a\odot d_ix$ and $s_i(a\odot
x)=a\odot s_ix$ for all $i$ and $x\in \mathcal M(S^n).$
\end{cor}

\begin{example}\label{Example_F}
Let $F[-]:{\sf Sets}_* \to {\sf Sets}_*$ be the functor that sends
a pointed set $X$ to the group with generators  $X$ and with one
relation $*=1,$ where $*$ is the base point of $X.$ The definition
of $F$ on morphisms is obvious. This functor has the natural
structure of a monad $(F,{\sf m},{\sf e}),$ where ${\sf
m}_X:F[F[X]]\to F[X]$ is given by 'concatenation' of 'strings of
strings' and ${\sf e}_X:X\to F[X]$ is the map to generators. The
map ${\sf e}_X$ universal point-preserving map from $X$ to a
group. Actually, $F[X]$ is a free group over $X\setminus \{*\},$
but this definition is not so good for morphisms. The extension of
$F[-]$ on the category of simplicial sets
$$F[-]:\mathcal S_*\longrightarrow \mathcal S_*$$
with the natural morphisms ${\sf m}_X:F[F[X]]\to F[x]$ and ${\sf
e}_X:X\to F[X]$
 is the monad of {\bf Milnor's construction}. An $F[-]$-algebra is a {\bf simplicial group}. It follows that for any simplicial group $G$ the Kleisli composition gives the maps
 $$\odot : G_n\times F[\Delta[n]^+]_m \longrightarrow G_m, \hspace{1cm} \odot : Z_nG\times F[S^n]_m \longrightarrow G_m.$$
Similarly one can define the abelian version of this monad
$\mathbb Z[-],$ Lie ring (over $\mathbb Z$) version $\mathcal
L[-]$ and $p$-restricted Lie algebra over $\mathbb F_p$ version
$\mathcal L^{[p]}[-].$
\end{example}

Assume now that $\mathcal M$ is a monad with 'a natural structure
of a group' on $\mathcal MX.$ More precisely, assume that we fix a
morphism of monads $F[-]\to \mathcal M.$ In particular, any
$\mathcal M$-algebra $A$ has a structure of a simplicial group. It
follows that the set $\mathcal S_*(X,\mathcal MY)$ has a structure
of a group given by $(g_1\cdot g_2)(x)=g_1(x)\cdot g_2(x).$ Using
that $\mathcal Mf$ and ${\sf m}^A$ are momorphisms of simplicial
groups, we get
\begin{equation}\label{eq_right_distributivity}
f\odot (g_1\cdot g_2)=(f\odot g_1)\cdot (f\odot g_2)
\end{equation}
for any $f\in \mathcal S_*(Y,A)$ and $g_1,g_2\in \mathcal
S_*(X,\mathcal MY).$ Indeed, \begin{multline*} (f\odot (g_1\cdot
g_2))(x)={\sf m}^A(\mathcal Mf( (g_1(x)\cdot g_2(x))))=\\ {\sf
m}^A(\mathcal Mf( g_1(x))) \cdot {\sf m}^A(\mathcal Mf(
g_2(x)))=((f\odot g_1)\cdot (f\odot g_2))(x).\end{multline*} Note
that in general $(f_1\cdot f_2)\odot g$ is not equal to $(f_1\odot
g)\cdot (f_2\odot g).$

Since an $\mathcal M$-algebra $A$ is a simplicial group, it is a
fibrant simplicial set. Then $[Y,A]=\mathcal S_*(Y,A)/\sim,$ where
$\sim$ is the homotopy equivalence. It is easy to see that if
$\{h_i\}$ is a simplicial homotopy between simplicial maps
$g,g'\in \mathcal S_*(X,A),$ then $\{\mathcal Mh_i\}$ is a
homotopy between simplicial maps $\mathcal Mg, \mathcal
Mg'\in\mathcal S_*(\mathcal MY, \mathcal MA),$ and hence,
$\mathcal M$ induces a well-defined map $[Y,A]\to [\mathcal
MY,\mathcal MA].$ It follows that in this case the Kleisli
composition is well-defined on the level of the homotopy category:
$$\odot: [Y,A]\times [X,\mathcal MY]\longrightarrow [X,A].$$
If we take $X=S^n,Y=S^m,$ then we get the map
\begin{equation}\label{eq_homotopy_M-comp}
\odot :\pi_nA \times \pi_n\mathcal MS^n\longrightarrow \pi_mA.
\end{equation}
The equation \eqref{eq_M-comp2} implies that an element $x\in
\pi_m\mathcal MS^n$ a gives a map $-\odot x:\pi_n A\to \pi_m,$
which is natural by $A.$ In other words, if $\alpha:A\to B$ is a
morphism of $\mathcal M$-algebras, then
\begin{equation}\label{eq_odot_x_nat}
(\pi_m\alpha)(a\odot x)=(\pi_n\alpha) a\odot x
\end{equation}
for any $a\in \pi_nA.$

Despite that the left distributivity does not hold in general for
$\odot$, we have the following lemma that partially substitute the
left distributivity.
\begin{lemma}\label{Lemma_part_left_distr} Let $\mathcal M=F[-]$, $G$ is a simplicial group and $H\triangleleft G$ is its normal simplicial subgroup. Denote by $\iota\mathcal S_*(-,H)$ the image of $\mathcal S_*(-,H)$ in $\mathcal S_*(-,G).$ If
 $f\in \mathcal S_*(Y,G),$ $h\in \iota \mathcal S_*(Y,H),$  $g\in \mathcal S_*(X,F[Y]),$ then
$$((f\cdot h)\odot g)\cdot (f\odot g)^{-1}\in  \iota \mathcal S_*(X,H),$$
$$((h\cdot f)\odot g)\cdot (f\odot g)^{-1}\in  \iota \mathcal S_*(X,H).$$
\end{lemma}
\begin{proof}
Consider the canonical projection $\alpha:G\to G/H.$ Then
$\alpha\circ h$ is trivial and $\alpha\circ (f\cdot h)=\alpha\circ
f.$ It follows that $\alpha_*((f\cdot h)\odot g)=\alpha\circ
(f\cdot h)\odot g=\alpha\circ f\odot g=\alpha_*( f\odot g).$ Using
that the sequence $1\to \mathcal S_*(X,H)\to \mathcal S_*(X,G)
\overset{\alpha_*}\to \mathcal S_*(X,G/H)$ is exact, we get the
required statement.
\end{proof}

\subsection{Suspension homomorphism}

For a pointed simplicial set $X$ we always assume that $X_{-1}=*$
and  $d_0:X_0\to X_{-1}$ is the unique map. There are two versions
of the suspension and the loop space of a pointed simplicial set
that are opposite to each other. We prefer the following (as in
\cite{Kan-Whitehead}). For $X\in \mathcal S_*$ we denote by
$\Sigma X$ the a pointed simplicial set such that
$$(\Sigma X)_n=\{ (x,k) \mid x\in X_{n-k} \text{ \  and \  } 0\leq k\leq n+1\}/\sim,$$
where $\sim$ is identifying of points $(*,k)$ and $(x,0)$ with the
base point $(*,n+1).$ For $x\in X_{n-k}$ we define
\begin{align*} & d_i(x,k)= \left\{\begin{array}{ll}
(d_ix,k), & 0\leq i\leq n-k\\
(x,k-1), & n-k<i \leq n
\end{array} \right.\\
& s_i(x,k)= \left\{\begin{array}{ll}
(s_ix,k), & 0\leq i\leq n-k\\
(x,k+1), & n-k<i \leq n
\end{array} \right.
\end{align*}
By $P X$ we denote the simplicial path space which is given by
$$(P X)_n=\{ x\in X_{n+1}\mid (d_0)^{n+1}x=*  \},$$
and $d_i,s_i$ for $P X$ are induced by $d_i$ and $s_i$ for $X.$
The simplicial loop space $\Omega X$ is a simplicial subset of the
path space $PX,$ which is given by
$$(\Omega X)_n=\{x\in X_{n+1}\mid (d_0)^{n+1}x=* \text{ and } d_{n+1}x=*\}.$$
These functors are adjoint $\mathcal S_*(\Sigma X,Y)\cong \mathcal
S_*(X,\Omega Y),$ and hence
 $\mathcal S_*(\Sigma X,Y)$ is isomorphic to the set
\begin{multline*}\Big\{ \{H:X_n\to Y_{n+1}\}_{n\geq 0} \mid\\   d_i H=Hd_i,
s_iH=Hs_i, d_{n+1}H=*, (d_0)^{n+1} H=*,\ 0\leq i\leq
n\Big\},\end{multline*} where the isomorphism is given by
$f\mapsto H_f$, where $H_f(x)=f(x,1)$. Let $\mathcal M :{\sf
Sets}_* \to {\sf Sets}_*$ be a functor such that $\mathcal
M(*)=*.$ We extend it to the category of pointed simplicial sets
$\mathcal S_*.$
 Then there is a map
$$\Phi^{\mathcal M}_{X,Y}:\mathcal S_*(\Sigma X,Y)\longrightarrow \mathcal S_*(\Sigma\mathcal MX,\mathcal MY)$$
that sends the morphism corresponding to a collection $\{H:X_n\to
Y_{n+1}\}$ to the morphism corresponding to the collection
$\{\mathcal M H: \mathcal MX_n\to \mathcal MY_{n+1} \}.$ Consider
the map ${\bf s}_X^{\mathcal M}:\Sigma\mathcal MX \longrightarrow
\mathcal M \Sigma X$ which is given by ${\bf s}_X^{\mathcal
M}=\Phi^{\mathcal M}_{X,\Sigma X}({\sf id}_{\Sigma X}).$ It gives
a natural transformation
\begin{equation}\label{eq_s_bold}
{\bf s}^{\mathcal M}: \Sigma \mathcal M \longrightarrow \mathcal M
\Sigma
\end{equation}
It is easy to check that
\begin{equation}\label{eq_Phi_sigma}
\Phi_{X,Y}^{\mathcal M}f=\mathcal Mf\circ {\bf s}_X^{\mathcal M}.
\end{equation}
Moreover, if $\mathcal M$ and $\mathcal N$ are two such
endofunctors, then
\begin{equation}\label{eq_Phi_Phi}
\Phi^{\mathcal N}_{\mathcal MX,\mathcal MY}\circ \Phi^{\mathcal
M}_{X,Y}=\Phi^{\mathcal N\mathcal M}_{X,Y},
\end{equation}
and if $\varphi:\mathcal M\to \mathcal N$ is a natural
transformation, then the following diagram is commutative
\begin{equation}\label{eq_sigma_nat}
\xyma{
  \Sigma \mathcal M \ar@{->}[r]^{\Sigma\varphi} \ar@{->}[d]^{{\bf
s}^{\mathcal M}} & \Sigma \mathcal N\ar@{->}[d]^{{\bf
s}^{\mathcal N}}\\
  \mathcal M\Sigma \ar@{->}[r]^{\varphi\Sigma} & \mathcal N\Sigma}
\end{equation}
The equalities \eqref{eq_Phi_sigma} and \eqref{eq_Phi_Phi} imply
the following
\begin{equation}\label{eq_sigma_MN}
{\bf s}^{\mathcal N\mathcal M}=\mathcal N{\bf s}^{\mathcal M}\circ
{\bf s}^{\mathcal N}\mathcal M.
\end{equation}
The map
$$\sigma  : \mathcal S_*(X,\mathcal MY) \longrightarrow \mathcal S_*(\Sigma X,\mathcal M\Sigma Y)$$
given by $\sigma f={\bf s}_Y^{\mathcal M}\circ \Sigma f$ is called
{\bf suspension homomorphism} with respect to $\mathcal M.$

\begin{lemma}\label{L}
Let $(\mathcal M,{\sf m}, {\sf e})$ be a monad on the category ${\sf Sets}_*$ such that $\mathcal M(*)=*.$ Then for any $X,Y,Z\in \mathcal S_*$ and $f\in \mathcal S_*(Y,\mathcal M Z),$ $g\in \mathcal S_*(X,\mathcal MY)$ the following holds
$$\sigma f \odot \sigma g=\sigma (f \odot g).$$
\end{lemma}
\begin{proof} Denote ${\bf s}={\bf s}^{\mathcal M}$
Using \eqref{eq_sigma_MN}, we get ${\bf s}^{\mathcal M^2}=\mathcal
M {\bf s}\circ {\bf s}\mathcal M.$ Since ${\sf m}:\mathcal M^2 \to
\mathcal M$ is a natural transformation, using
\eqref{eq_sigma_nat}, we get ${\sf m}\Sigma\circ{\bf s}^{\mathcal
M^2} = {\bf s}\circ\Sigma {\sf m}.$ Combining these equalities, we
get ${\sf m} \Sigma \circ \mathcal M{\bf s} \circ {\bf s} \mathcal
M={\bf s}\circ\Sigma{\sf m}.$ Using this equality and the fact
that ${\bf s}$ is a natural transformation, we obtain
\begin{multline*}\sigma f\odot \sigma g={\sf m}_{\Sigma Z}\circ \mathcal
M{\bf s}_Z\circ \mathcal M\Sigma f\circ {\bf s}_Y\circ \Sigma g=\\
{\sf m}_{\Sigma Z}\circ \mathcal M{\bf s}_Z\circ \mathcal M\Sigma
f\circ {\bf s}_Y\circ \Sigma g={\sf m}_{\Sigma Z}\circ \mathcal M
{\bf s}_Z \circ {\bf s}_{\mathcal MZ} \circ \Sigma\mathcal M
f\circ \Sigma g=\\ {\bf s}_Z\circ \Sigma{\sf m}_Z \circ
\Sigma\mathcal M f\circ \Sigma g={\bf s}_Z\circ \Sigma({\sf m}_Z
\circ \mathcal M f\circ g)=\sigma (f\odot g).\end{multline*}
\end{proof}

\begin{remark}
Lemma \ref{L} means that $\Sigma$ induces a functor on the Kleisli
category associated with the monad $(\mathcal M,{\sf m},{\sf e}).$
\end{remark}

Suspension homomorphism can be defined on the level of the
homotopy category $\sigma :[X,\mathcal M Y]\to [\Sigma X,\mathcal
M\Sigma Y]$ by the same formula. If $X=S^m,Y=S^n$ we get the map
$$\sigma :\pi_m\mathcal M(S^n)\longrightarrow \pi_{m+1}\mathcal M(S^{n+1}).$$

\begin{remark} It can be checked that our definition of suspension
homomorphism coincides with the one given in \cite{Dold-Puppe} if
$\mathcal MX=\mathscr M(\mathbb Z[X]),$ for a functor $\mathscr M:{\sf Ab}\to {\sf Ab}\xrightarrow{\rm forg.} {\sf Sets}_*.$
\end{remark}

\subsection{The section $\tau: \mathcal M(S^n)_m\to \mathcal M (\Delta[n]^+)_m$}\label{subsection_the_section}

We denote by $ {\bf i}_n$  the non-degenerate $n$-simplex in
$\Delta[n]$ and its image in $S^n.$ Consider the map $\tau:
(S^n)_m\longrightarrow (\Delta[n]^+)_m$ given by $\tau(s_{i_1}\dots
s_{i_{m-n}} {\bf i}_n)=s_{i_1}\dots s_{i_{m-n}} {\bf i}_{n}$ and $\tau(*)=*.$ It
is {\bf not} a simplicial map but it commutes with degeneracies
and it is a section of the projection ${\sf
pr}_m:(\Delta[n]^+)_m\twoheadrightarrow (S^n)_m.$

Let  $(\mathcal M,{\sf m},{\sf e})$ be a monad on the category
${\sf Sets}_*$ such that $\mathcal M(*)=*,$ and as usual we extend
it on the category of pointed simplicial sets $\mathcal S_*.$
 Applying $\mathcal M$ to $\tau:(S^n)_m\to (\Delta[n]^+)_m$ we get a map
$$\tau: \mathcal M(S^n)_m\to \mathcal M (\Delta[n]^+)_m$$
that we denote by the same symbol. Consider the adjoint map ${\bf
s}'_X:\mathcal MX \to \Omega \mathcal M\Sigma X$ to the map
\eqref{eq_s_bold}. If we take $X=S^n$, the simplicial homomorphism
${\bf s}'_{S^n}$ induces the following morphism of $\mathcal
M$-algebras
$${\bf s}:\mathcal M(S^n)_m\longrightarrow \mathcal M(S^{n+1})_{m+1}.$$

\begin{lemma}\label{Lemma_tau_s} Let $A$ be an $\mathcal M$-algebra and $x\in \mathcal M(S^n)_m$. If $a\in A_{n+1}$ such that $d_ia=1$ for $0\leq i\leq n,$ then $d_i(a\odot \tau {\bf s}x)=a\odot \tau {\bf s}d_ix$ for $0\leq i\leq n$ and $$d_{m+1}(a \odot \tau {\bf s} x)=d_{n+1}a\odot \tau  x.$$
\end{lemma}
\begin{proof}
The path space $PA$ is still a simplicial $\mathcal M$-algebra,
and hence we can consider the unique momorphism of simplicial
$\mathcal M$-algebras $\psi_a:\mathcal M(S^n)\to PA$ such that
$\psi_a({\sf e}{\bf i}_n)=a$ (see Corollary \ref{Corolary_Sph}).
Prove that $\psi_a(x)=a\odot \tau {\bf s}x,$ where $a$ is
considered as an element of $A_{n+1}$ and $x\in \mathcal
M(S^n)_m.$ Since both the maps $\mathcal M(S^n)_m\to (PA)_m$ are
morphisms of $\mathcal M$-algebras, it is enough to check that the
compositions $(S^n)_m\overset{{\sf e}}\to\mathcal M(S^n)_m\to
(PA)_m$ are equal. Indeed, for $y=s_{i_1}\dots s_{i_{m-n}}\in
(S^n)_m$ we have $\psi_a({\sf e}y)=s_{i_1}\dots
s_{i_{m-n}}a=a\odot \tau {\bf s}{\sf e}y.$ Thus
$\psi_a=a\odot\tau{\bf s}(-),$ and hence $a\odot\tau{\bf s}(-)$ is
a simplicial map. It follows that $d_i(a\odot \tau {\bf
s}x)=a\odot \tau {\bf s}d_ix$ for $0\leq i\leq n.$

 Prove the last equality. Since both maps $d_{m+1}(a\odot \tau {\bf s}(-))$ and $d_{n+1}a\odot \tau (-)$
 are morphisms of $\mathcal M$-algebras, it is enough to prove for $x={\sf e}y,$ where $y\in (S^n)_m.$
 For an element of the form $y=s_{i_1}\dots s_{m-n} {\bf i_n}$ we have that $i_j\leq m-j+1.$ It follows that
 $d_{m+1}(a\odot \tau {\bf s}{\sf e} y)= d_{m+1}s_{i_1}s_{i_2}\dots s_{i_{m-n}} a=$ $s_{i_1}d_{m}s_{i_2}\dots s_{i_{m-n}}a=$
 ${}\dots{}=$ $s_{i_1}s_{i_2}\dots s_{i_{m-n}}d_{n+1}a$ $=d_{n+1}a\odot \tau {\sf e}y.$
\end{proof}

\section{Multiplicative unstable spectral sequence}

\subsection{General lemma about simplicial groups}

 If $G$ is a simlicial group, we denote by $NG$ its Moore complex, where $N_nG=\bigcap_{i\geq 1} {\rm Ker}\: d_i$ and $d:N_nG\to N_{n-1}G$ is induced by $d_0.$ By $Z_nG$ we denote the group of Moore cycles $Z_nG=\{g\in G_n\mid d_ig=1 \text{ for } 1\leq i\leq n\}.$ If $H$ is a normal simplicial subgroup of $G,$ we get the boundary map $\delta:\pi_n(G/H)\to \pi_{n-1} H.$ If $g\in N_nG$ such that $gH\in Z_n(G/H),$ then $\delta([gH])=[d_0g].$ The following lemma is a generalisation of this equality.

\begin{lemma}\label{Lemma_bound_non_Moore} Let $G$ be a simplicial group, $H$ be its normal subgroup,  $\delta:$ $\pi_n(G/H)$ $\to$  $\pi_{n-1}H$ be the boundary map,  $g\in G_n$ such that $d_ig\in Z_{n-1}H$ for all $0\leq i\leq n.$ Then
$$\delta([gH])=\sum_{i=0}^n (-1)^i[d_ig].$$
\end{lemma}
\begin{proof}
Consider the sequence of elements $g_n, g_{n-1}, \dots, g_0,$ where $g_n=g$ and $g_j= g_{j+1}\cdot(s_{j}d_{j+1}g_{j+1})^{-1} .$  We claim that $d_ig_j=1$ for $j<i\leq n,$ $d_ig_j=d_{i}g $ for $i<j,$ and
$$d_jg_j=d_{j+\pi(0)}g^{(-1)^{\pi(0)}}\cdot d_{j+\pi(1)}g^{(-1)^{\pi(1)}} \cdot {}\dots{} \cdot d_{j+\pi(n-j)}g^{(-1)^{\pi(n-j)}},$$
for some permutation $\pi=\pi_j$ of the set $\{0,1,\dots, n-j\}.$ Moreover, we claim that $g_jH=gH.$

Prove it using descending induction by $j.$
It is obvious if $j=n.$ Assume that it is true for $j+1$ and prove for $j.$ If $j+1< i\leq n$ then
$$d_ig_j=d_ig_{j+1} \cdot d_is_jd_{j+1}g_{j+1}^{-1}=d_is_jd_{j+1}g_{j+1}^{-1}=s_jd_{i-1}d_{j+1}g_{j+1}^{-1}=s_jd_{j+1}d_ig_{j+1}^{-1}=1.$$
If $i=j+1,$ then
$$d_{j+1}g_j=d_{j+1}g_{j+1} \cdot d_{j+1}s_jd_{j+1}g_{j+1}^{-1}=d_{j+1}g_{j+1} \cdot d_{j+1}g_{j+1}^{-1}=1.$$
If $i<j,$  using that $d_is_jd_{j+1}=s_{j-1}d_id_{j+1}=s_{j-1}d_jd_i$ and that  $d_ig \in Z_{n-1}H,$ we get
$$d_ig_j=d_ig_{j+1} \cdot d_is_jd_{j+1}g_{j+1}^{-1}=d_ig \cdot s_{j-1}d_{j}d_ig^{-1}=d_ig.$$
If $i=j,$ we get
$$d_jg_j=d_{j}g_{j+1} \cdot d_{j}s_jd_{j+1}g_{j+1}^{-1}=d_{j}g \cdot d_{j+1}g_{j+1}^{-1}=$$
$$=d_{j}g \cdot d_{j+1+\pi(n-j-1)}g^{(-1)^{\pi(n-j-1)+1}} \cdot {}\dots{} \cdot d_{j+1+\pi(0)}g^{(-1)^{\pi(0)+1}}=$$
$$=d_{j+\tilde \pi(0)}g^{(-1)^{\tilde \pi(0)}}\cdot d_{j+\tilde \pi(1)}g^{(-1)^{\tilde \pi(1)}} \cdot {}\dots{} \cdot d_{j+\tilde \pi(n-j)}g^{(-1)^{\tilde \pi(n-j)}},$$
where $\tilde \pi(0)=0$ and $\tilde \pi(k)=\pi(n-j-k)+1.$ Since $\pi$ is a permutation of the set $\{0,\dots, n-j-1\},$  $\tilde \pi$ is a permutation of the set $\{0,\dots, n-j\}.$ Using the above presentations of $d_ig_j,$ we see that $d_ig_j\in Z_{n-1}H.$ In particular, $s_jd_{j+1}g_{j+1}\in H_n,$ and hence, $g_jH=g_{j+1}H=\dots = gH.$

Therefore, we have that $g_0\in N_nG,$ $g_0H=gH$ and $$d_0g_0=d_{ \pi(0)}g^{(-1)^{ \pi(0)}}\cdot d_{ \pi(1)}g^{(-1)^{ \pi(1)}} \cdot {}\dots{} \cdot d_{\pi(n)}g^{(-1)^{ \pi(n)}},$$ for some permutation $\pi$ of the set $\{0,\dots,n\}.$ It follows that $\delta([gH])=\delta([g_0H])=[d_0g_0]=\sum_{i=0}^n (-1)^i[d_ig].$

\end{proof}

\subsection{Spectral sequence given by a filration of a group-like monad}

Let $(\mathcal M,{\sf m},{\sf e})$ be a monad over $F[-]$ (with a
fixed morphism of monads $F[-]\to \mathcal M$) on the category of
${\sf Sets}_*,$  where $F[-]$ is the monad from Example
\ref{Example_F}. Assume that $\mathcal M$ is filtered  by a
sequence of subfunctors
$$\mathcal MX=\mathfrak F_0X \supseteq \mathfrak F_1X \supseteq \mathfrak F_2X \supseteq \dots $$
such that $\mathfrak F_iX$ is a normal subgroup of $\mathcal MX$
and the multiplication morphism ${\sf m}:\mathcal M^2\to \mathcal
M$ induces a morphism
\begin{equation}\label{eq_filt_prop}
{\sf m}: \mathfrak F_i\circ \mathfrak F_j\longrightarrow \mathfrak
F_{i+j}.
\end{equation}
This filtration induces a filtration on the extended monad
$\mathcal M:\mathcal S_*\to \mathcal S_*$ with the same
properties.

We put $\mathfrak L_iX:=\mathfrak F_iX/\mathfrak F_{i+1}X,$
$\mathfrak L_*X=\bigoplus_{i\geq 0} \mathfrak L_iX$ and $\mathfrak
F_*X= \oplus_{i\geq 0} \mathfrak F_iX.$ It is easy to see that the
structure of monad on $\mathcal M$ induces a structure of a monad
on $\mathfrak L_*.$ If $A$ is a simplicial $\mathcal M$-algebra
with the structure map ${\sf m}^A:\mathcal MA\to A$ the filtration
$\mathfrak{F}$ defines a filtration on $A$ given by
$\mathfrak{f}_iA={\sf m}^A(\mathfrak{F}_iA).$ We put
$\mathfrak{l}_iA=\mathfrak{f}_iA/\mathfrak{f}_{i+1}A,$ $\mathfrak
f_*A=\bigoplus_{i\geq 0} \mathfrak{f}_iA$ and $\mathfrak
l_*A=\bigoplus_{i\geq 0} \mathfrak{l}_iA.$ Then $\mathfrak{l}_*A$
is an algebra over the monad $\mathfrak{L}_*.$
 Further, we denote by  $C(A,\mathcal M)$ the associated homotopy exact couple $(D(A,\mathcal M),B(A,\mathcal M),\alpha,\beta,\delta),$ where $D(A,\mathcal M)= \pi_*\mathfrak f_*A$ and $B(A,\mathcal M)=\pi_* \mathfrak l_*A.$ The associated spectral sequence we denote by $E(A,\mathcal M).$ We are using the following indexing $E^1_{i,j}(A,\mathcal M)=\pi_j\mathfrak l_iA,$ and hence the differential on the $r$th page acts as follows $d^r:E^r_{k,l}(A,\mathcal M)\to E^r_{k+r,l-1}(A,\mathcal M).$

If $A=\mathcal MX$ for a simplicial set $X,$ we denote
$C(X,\mathcal M):=C(\mathcal MX,\mathcal M)$ and $E(X,\mathcal
M)=E(\mathcal MX,\mathcal M).$ The suspension homomorphism induces
a morphism of exact couples $\sigma:C(X,\mathcal M)\to C(\Sigma
X,\mathcal M).$  Then we have $E^1(X,\mathcal M)=\pi_*\mathfrak
L_*X$ and the suspension morphism induces a morphism of spectral
sequences
 $$\sigma:E(X,\mathcal M)\to E(\Sigma X,\mathcal M).$$

The spectral sequence is obviously functorial by $\mathcal M$ i.e.
if $\mathcal M\to \mathcal M'$ is a morphism of filtered monads
over $F[-]$ then we get a morphism of spectral sequences
 $$E(X,\mathcal M)\longrightarrow E(X,\mathcal M').$$

Consider the direct sum of spectral sequences $E(S^*,\mathcal
M)=\bigoplus_{n\geq 2} E(S^n,\mathcal M).$ If we denote
$\pi_*\mathfrak L_*(S^*)=\bigoplus_{n\geq 2} \pi_*\mathfrak
L_*(S^n)$, we get
$$E^1(S^*,\mathcal M)=\pi_*\mathfrak L_*(S^*).$$
The Kleisli composition with respect to $\mathcal L$ induces an
operation
$$\odot: \pi_*\mathfrak l_*A \times \pi_*\mathfrak L_*(S^*) \longrightarrow \pi_*\mathfrak l_*A$$
that vanishes on non-composable homogenious elements. Using
\eqref{eq_right_distributivity}, we get $x\odot (y+z)=x\odot
y+x\odot z,$ but we {\bf do not} state that the map is bilinear,
and so we do not assume that $\pi_* \mathfrak L_*(S^*)$ is a ring
with respect to $\odot$. The suspension homomorphisms give an
endomorphism
$$\sigma: \pi_* \mathfrak L_*(S^*)\longrightarrow \pi_*\mathfrak L_*(S^*)$$
such that $\sigma(a+b)=\sigma a+\sigma b$ and $\sigma(a\odot
b)=\sigma a\odot \sigma b.$ Moreover, we get the suspension
endomorphism of the spectral sequence
$$\sigma: E(S^*,\mathcal M)\longrightarrow E(S^*,\mathcal M)$$
that consists of components $\sigma: E^r_{k,l}(S^n,\mathcal M)\to
E^r_{k,l+1}(S^{n+1},\mathcal M).$

\begin{theorem}\label{Theorem_E(S^*,M)} Let $A$ be a $\mathcal M$-algebra.
Then the Kleisli composition $\odot$ induces a well-defined  map $$\odot: E^r(A,\mathcal M)\times E^r(S^*,\mathcal M)\longrightarrow E^r(A,\mathcal M),$$ and
for elements $a\in E^r_{k,l}(A,\mathcal M)$ and $b\in
E^r_{s,t}(S^*,\mathcal M)$ the following Leibniz-like rule
holds
\begin{equation}\label{eq_odot_der}
d^r(a\odot \sigma b)=(-1)^{l+t} d^r a \odot b+a\odot d^r \sigma b.
\end{equation}
\end{theorem}
\begin{proof}
For the sake of simplicity we put $E(A)=E(A,\mathcal M)$ and
$E(S^*)=E(S^*,\mathcal M).$ Denote
$$Z^r_{k,l}(A)={\rm Im}\Big(\pi_l(\mathfrak f_kA/\mathfrak f_{k+r}A) \longrightarrow \pi_l\mathfrak l_kA\Big)={\rm Ker}\Big( \pi_l\mathfrak{l}_kA \overset{\delta}\longrightarrow \pi_{l-1}(\mathfrak{f}_{k+1}A/\mathfrak{f}_{k+r}A)\Big),$$
where $\delta$ is the boundary map associated with the short exact
sequence  $$1\to\mathfrak{f}_{k+1}A/\mathfrak{f}_{k+r}A \to
\mathfrak{f}_{k}A/\mathfrak{f}_{k+r}A \to  \mathfrak{l}_{k}A\to
1,$$ and
\begin{multline*}
B^r_{k,l}(A)={\rm Im}\Big(\pi_{l+1}(\mathfrak{f}_{k-r+1}A/\mathfrak{f}_{k}A)\overset{\delta}\longrightarrow \pi_l\mathfrak{l}_kA\Big)=\\
{\rm Ker}\Big(\pi_l\mathfrak{l}_kA\longrightarrow \pi_{l}(\mathfrak{f}_{k-r+1}A/\mathfrak{f}_{k+1}A)  \Big),
\end{multline*}
where $\delta$ is the boundary map associated with the short exact
sequence $$1\to \mathfrak{l}_kA\to
\mathfrak{f}_{k-r+1}A/\mathfrak{f}_{k+1}A \to
\mathfrak{f}_{k-r+1}A/\mathfrak{f}_{k}A\to 1.$$ Then
$E^r_{k,l}(A)=Z^r_{k,l}(A)/B^r_{k,l}(A).$

Prove that $\odot$ is well-defined on the level of $Z^r.$ Using
the property \eqref{eq_filt_prop}, we get that the map ${\sf
m}:\mathcal M^2\to \mathcal M$ induces the map
$$\mathfrak F_k/\mathfrak F_{k+r}\circ \mathfrak F_s/\mathfrak F_{s+r} \longrightarrow \mathfrak F_{k+s}/\mathfrak F_{k+s+r}.$$
Hence we get a well-defined structure of a monad on
$\bigoplus_{j\geq 0} \mathfrak{F}_{j}/\mathfrak{F}_{j+r}.$
Consider elements $a\in Z^r_{k,l}(A)$ and $b\in Z^r_{s,t}(S^m).$
If $l\ne m,$ the composition vanishes, so we assume that $l=m.$ By
definition of $Z^r,$  there exist elements  $\tilde a\in
\pi_m(\mathfrak f_kA/\mathfrak f_{k+r}A)$ and $\tilde b\in \pi_t
(\mathfrak{F}_sS^m/\mathfrak{F}_{s+r}S^m)$ that lift $a$ and
$b.$ Then their Kleisli composition $\tilde a\odot \tilde b\in
\pi_{t}(\mathfrak{f}_{k+s}A/\mathfrak{f}_{k+s+r}A)$ with respect
to the monad $\bigoplus_{j\geq 0}
\mathfrak{F}_{j}/\mathfrak{F}_{j+r}$ lifts the Kleisli composition
$a\odot b\in \pi_{t}\mathfrak l_{k+s}A,$ and hence $a\odot b\in
Z^r_{k+s,t}(A).$ Thus $\odot$ is well-defined on the level of
$Z^r.$

In order to prove that $\odot$ is well defined on the level of
$E^r$ we have to prove that for elements $x\in B^r_{k,m}(A)$ and
$y\in B^r_{s,t}(S^m)$ we have $$(a+x)\odot (b+y)-a\odot b \in
B^r_{k+s,t}(A).$$ Using the right distributivity
\eqref{eq_right_distributivity}, we get $$(a+u)\odot (b+v)-a\odot
b=(a+u)\odot b+(a+u)\odot v-a\odot b.$$ Hence, it is enough to
prove that $a\odot v\in B^r_{k+s,t}(A)$ for any $a\in
Z^r_{k,m}(A)$ and $v\in B^r_{s,t}(S^m),$ and $(a+u)\odot b-a\odot
b \in B^r_{k+s,t}(A)$ for any $a\in Z^r_{k,m}(A),$ $b\in
Z^r_{s,t}(S^m)$ and $u\in B^r_{k,l}(A).$

Prove that  $a\odot v\in B^r_{k+s,t}(A)$ for any $a\in
Z^r_{k,m}(A)$ and $v\in B^r_{s,t}(S^m).$ Consider a lifting
$\tilde a\in Z_m(\mathfrak{f}_kA/\mathfrak{f}_{k+r}A)$ of $a$ and
an element  $\tilde v'\in
N_t(\mathfrak{F}_{s-r+1}(S^m)/\mathfrak{F}_{s+1}(S^m))$ such that
$\tilde v=d_0\tilde v'$ lifts $v.$ Then by Corollary \ref{Corolary_Sph}
$d_i(\tilde a\odot \tilde v')=\tilde a\odot d_i\tilde v',$ where
$\odot$ is the Kleisli composition with respect to
$\bigoplus_{j\geq 0} \mathfrak{F}_{j}/\mathfrak{F}_{j+r}.$ Thus
$d_i(\tilde a\odot \tilde v')=1$ for $1\leq i\leq t+1$ and
$d_0(\tilde a\odot \tilde v')=\tilde a\odot \tilde v$ lifts the composition $a\odot v.$
Hence $a\odot v\in B^r_{k+s,t}(A).$

Prove that $(a+u)\odot b-a\odot b \in B^r_{k+s,t}(A)$ for any
$a\in Z^r_{k,m}(A),$ $b\in Z^r_{s,t}(S^m)$ and $u\in
B^r_{k,m}(A).$ Consider the embedding $\alpha:\mathfrak l_kA\to
\mathfrak{f}_{k-r+1}A/\mathfrak{f}_{k+1}A.$ Then
$B^r_{k,m}(A)={\rm Ker}(\pi_m\alpha)$ and $B^r_{k,t}(A)={\rm
Ker}(\pi_t\alpha).$ Using \eqref{eq_odot_x_nat}, we get
\begin{multline*} (\pi_t\alpha)((a+u)\odot b-a\odot b)=(\pi_m\alpha)(a+u)\odot
b-(\pi_m\alpha)a\odot b=\\ (\pi_m\alpha)a\odot
b-(\pi_m\alpha)a\odot b=0.\end{multline*} And hence $(a+u)\odot
b-a\odot b \in B^r_{k+s,t}(A).$ Therefore, $\odot$ is well-defined
on the level of $E^r.$

Prove the Leibniz rule \eqref{eq_odot_der}. Let $a\in E^r_{k,l}(A)$ and $b\in E^r_{s,t}(S^m).$  If $m\ne l-1,$
then both parts of the equation \eqref{eq_odot_der} are equal to
zero, hence we can assume that $m=l-1.$ Consider $\tilde a\in
\mathfrak f_kA_l$ that lifts $a$ such that $d_i\tilde a=1$ for
$0\leq i\leq l-1$ and $d_l\tilde a\in \mathfrak f_{k+r}A_{l-1};$
and consider $\tilde b\in \mathfrak F_{s}(S^{l-1})_t$ such that
$d_i\tilde b=1$ for $1\leq i\leq t$ and $d_0\tilde b\in \mathfrak
F_{s+r}(S^{l-1})_{t-1}.$ Since the filtration $\mathfrak F$ is
defined on the level of pointed sets, we get $\tau {\bf s}\tilde
b\in \mathfrak F_{s}(\Delta[l])_{t+1}.$ Consider the element
$\tilde a\odot \tau {\bf s}\tilde b \in \mathfrak f_{k+s}A_{t+1}$
that lifts the element $a\odot \sigma b\in E_{k+s,t+1}(A).$ Then
by Lemma \ref{Lemma_tau_s} we obtain that $d_i(\tilde a\odot \tau
{\bf s}\tilde b)=1$ for $1\leq i\leq t-1,$ $d_0(\tilde a\odot \tau
{\bf s}\tilde b)=\tilde a\odot \tau{\bf s} d_0 \tilde b$ and
$d_{t}(\tilde a\odot \tau {\bf s}\tilde b)=d_l\tilde a\odot \tau
\tilde b.$ The element $\tilde a\odot \tau {\bf s} d_0\tilde b$
lifts the element $a\odot d^r \sigma \beta,$ using Lemma
\ref{Lemma_bound_non_Moore} we get that the  element $d_l\tilde
a\odot \tau \tilde b$ lifts the element $ (-1)^l d^ra\odot  b.$
Finally, using Lemma \ref{Lemma_bound_non_Moore} for the element
$\tilde a\odot \tau {\bf s}\tilde b$ we obtain  $d^r(a\odot
\sigma b)=(-1)^{t+l} d^r a \odot b+a\odot d^r\sigma b.$
\end{proof}

\begin{example}\label{Example_E_F_p} Consider the $p$-accelerated $p$-lower central series on the Milnor's construction
$F[X]=\gamma_1^{[p]}[X]\supseteq \gamma^{[p]}_p[X] \supseteq
\gamma^{[p]}_{p^2}[X] \supseteq \dots$ This filtration gives the
accelerated $p$-lower  central  series  spectral  sequence
$E(X,\mathbb F_p)$ \cite{6_authors}, \cite{Bousfield-Kan} for a
simplicial set $X$. If $G$ is a simplicial group (=simplicial
$F[-]$-algebra), we have the spectral sequence $E(G,\mathbb
F_p)$ which is associated to the $p$-accelerated $p$-lower central
series $G=\gamma_{p^0}^{[p]}(G)\supseteq
\gamma_{p^1}^{[p]}(G)\supseteq \gamma_{p^2}^{[p]}(G)\supseteq
\dots.$ We refer to this spectral sequence as a  {\bf $p$-LCS spectral sequence}. Then Theorem \ref{Theorem_E(S^*,M)} states that the
Kleisli composition induces a well-defined  map $$\odot:
E^r(G,\mathbb F_p)\times E^r(S^*,\mathbb F_p)\longrightarrow
E^r(G,\mathbb F_p),$$ and the property \eqref{eq_odot_der} holds.
In particular, the map $$\odot: E^r(X,\mathbb F_p)\times
E^r(S^*,\mathbb F_p)\longrightarrow E^r(X,\mathbb F_p)$$ is well
defined and  satisfies the property  \eqref{eq_odot_der}.
\end{example}
\begin{example}\label{Example_E_p}
Consider the $p$-accelerated (integral) lower central series
$F[X]=\gamma_1[X]\supseteq \gamma_p[X] \supseteq \gamma_{p^2}[X]
\supseteq \dots.$ We denote by $\tilde E(X,p)$ the associated
spectral sequence. Since $\gamma_{p^l}[X]\subseteq
\gamma_{p^l}^{[p]}[X]$, we get a morphism of spectral sequences:
$$\tilde E(X,p)\longrightarrow E(X,\mathbb F_p).$$
\end{example}
\begin{example}
Let $\mathbb Z_{(p)}=\{a/b\in \mathbb Q\mid b \text{ is prime to }
p\}.$ Denote by $F_{\mathbb Z_{(p)}}[-]:{\sf Sets}_*\to {\sf
Sets}_*$ the $\mathbb Z_{(p)}$-completion of the group $F[X]$ (see
\cite{Bousfield-Kan_compl}). The structure of a monad on $F[-]$
induces a structure of a monad on $F_{\mathbb Z_{(p)}}[-].$ Since
$\mathbb Z_{(p)}$-completion is an exact functor, we have
$\pi_*F_{\mathbb Z_{(p)}}[X]$ is the $\mathbb Z_{(p)}$-completion
of $\pi_{*+1}\Sigma X.$ Consider the $p$-accelerated lower central
series filtration on this monad. One could prove that the
associated spectral sequence convergence to the $\mathbb
Z_{(p)}$-completion of $\pi_{*+1}\Sigma X.$
\end{example}

\subsection{The lambda algebra for an odd prime $p$}

In this subsection we remind the notion of the lambda algebra for
odd $p$ (see \cite{6_authors} with corrections in
\cite{Bousfield-Kan}). Let $p$ be a fixed odd prime number. The
mod-$p$ lambda algebra ${}_{[p]}\Lambda=\Lambda$ is an $\mathbb
F_p$-algebra generated by elements $\lambda_i$ of degree
$2(p-1)i-1$ for $i\geq 1$ and elements $\mu_j$ of degree $2(p-1)j$
for $j\geq 0.$ If we consider the elements ${\tt a}(k,j), {\tt
b}(k,j)\in \mathbb F_p$ given by  \begin{align*} & {\tt
a}(k,j)=(-1)^{j+1}\binom{(p-1)(k-j)-1}{j},\\ & {\tt
b}(k,j)=(-1)^j\binom{(p-1)(k-j)}{j},\end{align*} and the elements
$N(k),N'(k)\in \mathbb Z$ given by $N(k)=\left\lfloor
k-\frac{k+1}{p}\right\rfloor,$ $N'(k)=\left\lfloor
k-\frac{k}{p}\right\rfloor,$ the ideal of relations of $\Lambda$
is generated by the relations
\begin{equation}\label{eq_ll}
\lambda_i\lambda_{pi+k}=\sum_{j=0}^{N(k)} {\tt a}(k,j)
\lambda_{i+k-j}\lambda_{pi+j},
\end{equation}
\begin{equation}\label{eq_lm}
\lambda_i\mu_{pi+k}=\sum_{j=0}^{N(k)} {\tt a}(k,j)
\lambda_{i+k-j}\mu_{pi+j}+\sum_{j=0}^{N'(k)} {\tt b}(k,j)
\mu_{i+k-j}\lambda_{pi+j}
\end{equation}
for $i\geq 1,k\geq 0,$ and the relations
\begin{equation}\label{eq_ml}
\mu_i\lambda_{pi+k+1}=\sum_{j=0}^{N(k)} {\tt a}(k,j)
\mu_{i+k-j}\lambda_{pi+j+1},
\end{equation}
\begin{equation}\label{eq_mm}
\mu_i\mu_{pi+k+1}=\sum_{j=0}^{N(k)} {\tt a}(k,j)
\mu_{i+k-j}\mu_{pi+j+1}
\end{equation}
for $i\geq 0,k\geq 0.$ The differential $\partial:\Lambda\to
\Lambda$ is given by
\begin{align*}
\partial\lambda_k =\sum_{j=1}^{N(k)}{\tt
a}(k,j)\lambda_{k-j}\lambda_j,\\ & \partial
\mu_k=\sum_{j=0}^{N(k)} {\tt
a}(k,j)\lambda_{k-j}\mu_j+\sum_{j=1}^{N'(k)}{\tt b}(k,j)
\mu_{k-j}\lambda_{j}.\end{align*}

Further by $\nu_i$ we denote an element of $\{\lambda_i,\mu_i\}.$
A monomial $\nu_{i_1}\dots \nu_{i_l}$ is said to be {\bf
admissible} if $i_{k+1}\leq pi_k-1$  whenever
$\nu_{i_k}=\lambda_{i_k}$ and if $i_{k+1}\leq pi_k$ whenever
$\nu_{i_k}=\mu_{i_k}.$ The set of admissible monomials is a basis
of $\Lambda$. The {\bf unstable  lambda algebra} $\Lambda(n)$ is a
dg-subalgebra of $\Lambda$ generated by admissible elements
$\nu_{i_1}\dots \nu_{i_l}$ such that $i_1\leq n.$ We denote by
$\Lambda(n)_m$ the subspace generated by monomials of degree $m$
in $\Lambda(n)$ and by $\Lambda(n)_{m,l}$ the vector space
generated by monomials of length $l$ in $\Lambda(n)_m.$ Then
$$\Lambda(n)=\bigoplus_{m,l}\Lambda(n)_{m,l}, \hspace{1cm} \Lambda(n)_m=\bigoplus_{l} \Lambda(n)_{m,l}.$$

Consider the left ideal of $\Lambda:$
\begin{equation}
\Lambda\lambda=\sum_i \Lambda \lambda_i.
\end{equation}
The set of all admissible monomials $\nu_{i_1}\dots \nu_{i_l}$
such that $\nu_{i_l}=\lambda_{i_l}$ forms a basis of
$\Lambda\lambda.$ Further we put
$$\Lambda\lambda(n)=\Lambda\lambda\cap \Lambda(n).$$ The set of
all admissible monomials $\nu_{i_1}\dots \nu_{i_l}$ such that
$\nu_{i_l}=\lambda_{i_l}$ and $\nu_{i_1}$ is one of the first $n$
letters in the sequence $\mu_0,\lambda_1,\mu_1,
\lambda_2,\mu_2,\dots$ forms a basis of $\Lambda\lambda(n).$ This
ideal is homogeneous and  inherits the bigradding:
$$\Lambda\lambda(n)=\bigoplus_{m,l} \Lambda\lambda(n)_{m,l}, \hspace{1cm} \Lambda\lambda(n)_m=\bigoplus_{l} \Lambda\lambda(n)_{m,l}.$$

\subsection{(Non)restricted Lie powers and the lambda algebra}

Denote by $\mathcal L[-]:{\sf Sets}_* \longrightarrow {\sf
Sets}_*$
 the functor that sends a pointed set $X$ to the unrestricted Lie ring (over $\mathbb Z$) with generators  $X$ and with one relation $*=0.$ For a prime $p$ similarly we define the functor $\mathcal L^{[p]}[-]:{\sf Sets}_*\to {\sf Sets}_*$ that sends $X$ to the restricted Lie algebra over $\mathbb F_p.$ They have natural structures of monad, and these monads can be extended to the category of simplicial sets:
$$\mathcal L[-], \mathcal L^{[p]}[-]:{\mathcal S}_* \longrightarrow {\mathcal S}_*.$$
These functors have standard gradings
$$\mathcal L[X]=\bigoplus_{i\geq 1} \mathcal L_i[X], \hspace{1cm} \mathcal L^{[p]}[X]=\bigoplus_{i\geq 1} \mathcal L^{[p]}_i[X].$$
The Kleisli composition induces operations on $\pi_*\mathcal
L[S^*]=\bigoplus_{n\geq 2} \pi_*\mathcal L[S^n]$ and
$\pi_*\mathcal L^{[p]}[S^*]=\bigoplus_{n\geq 2} \pi_*\mathcal
L^{[p]}[S^n].$

 The homotopy groups of $\mathcal L(S^n)$ can be described in terms of derived functors in sense of Dold-Puppe:
$$\pi_m\mathcal L_i[S^{n}]=L_m\mathscr L_i(\mathbb Z,n),\hspace{1cm} \pi_m\mathcal L_i^{[p]}[S^{n}]=L_m\mathscr L_i^{[p]}(\mathbb Z,n),$$
where $\mathscr L_i:{\sf Ab} \to {\sf Ab}$ is the functor of $i$th
Lie power (over $\mathbb Z$) and $\mathscr L_i^{[p]}:{\sf Ab} \to
{\sf Vect}(\mathbb F_p)$. It was proved in \cite{6_authors}, \cite{Bousfield-Kan} that,
if $i\ne p^l$ for some $l\geq 0$, then
$\pi_{m}(\mathcal L^{[p]}_i(S^{n}))=0,$ and there is an
isomorphism
\begin{equation}\label{eq_iso_L_p_Lambda}
\pi_{m}\mathcal L^{[p]}_{p^l}[S^{2n}]\cong {}_{[p]}\Lambda(n)_{m-2n,l}
\end{equation}
for any $l\geq 0.$ Moreover, the multiplication in the lambda algebra corresponds to
the Kleisli composition:
\begin{equation}\label{eq_comp_mult}
\xyma{\pi_{2m}\mathcal L_{p^l}^{[p]}[S^{2k}] \times \pi_{n}\mathcal L_{p^t}^{[p]}[S^{2m}]\ar@{->}[r]^>>>>>{\odot}\ar@{->}[d]^\cong &  \pi_n\mathcal L_{p^{l+t}}^{[p]}[S^{2k}]\ar@{->}[d]^\cong \\
\Lambda\lambda(k)_{m-k,l} \times
\Lambda\lambda(m)_{n-2m,t}  \ar@{->}[r]^>>>>>{\times} &
\Lambda\lambda(k)_{n-k,t+l}}.
\end{equation}

Daniella Leibowitz in \cite[Prop. 2.5]{Leibowitz} proves that if $i\ne p^l$ for some $l\geq 0$ and a
prime $p$, then $\pi_{m}(\mathcal L_i(S^{n}))=0,$ and there is an
isomorphism
\begin{equation}
\pi_{m}\mathcal L_{p^l}[S^{2n}]\cong {}_{[p]}\Lambda\lambda(
n)_{m-2n,l}
\end{equation}
for any odd prime $p$ and $l\geq 1.$ The simillar   Moreover, the embedding
$\mathcal L_{p^l}[X]\subseteq \mathcal L^{[p]}_{p^l}[X]$ induces
the commutative diagram
 \begin{equation}\label{pic_emb}
\xyma{\pi_{m}\mathcal L_{p^l}[S^{n}] \ar@{->}[r] \ar@{->}[d]^{\cong} & \pi_{m}\mathcal L^{[p]}_{p^l}[S^{n}]\ar@{->}[d]^\cong \\
 {}_{[p]}\Lambda\lambda(\lfloor n/2 \rfloor)_{m-n,l}\ar@{->}[r]^\subseteq & {}_{[p]}\Lambda(\lfloor n/2 \rfloor)_{m-n,l} }
\end{equation}

\subsection{Unstable multiplicative $p$-LCS spectral sequence}

Let $p$ be a fixed odd prime. Consider the $p$-LCS spectral sequence
$E(G,\mathbb F_p)$ (Example \ref{Example_E_F_p}) and denote
$E(X,\mathbb F_p)=E(F[X],\mathbb F_p),$ $E(S^*,\mathbb
F_p)=\bigoplus_{n\geq 2} E(S^n,\mathbb F_p).$ It is proved in
\cite{Rector} that
$$E^1(X,\mathbb F_p)=\pi_*\mathcal L_*^{[p]}[X]$$ and, if $X$ is connected and the homotopy groups $\pi_i(\Sigma X)$ are finitely generated, then
$$E(X,\mathbb F_p)\Rightarrow \pi_*(\Sigma X,p),$$
where $\pi_*(\Sigma X,p)$ is the quotient of $\pi_*(\Sigma X)$ by
the subgroup of elements of finite order prime to $p.$
In particular, if we take $X=S^{2n}$ and apply \eqref{eq_iso_L_p_Lambda}, we get
$$E^1(S^{2n},\mathbb F_p)={}_{[p]}\Lambda(n) \Rightarrow \pi_*(S^{2n+1},p).$$

\begin{theorem}\label{simplicial_theorem}
Let $G$ be a simplicial group.  Then the Kleisli composition
$\odot$ induces a well-defined  map $$\odot: E^r(G,\mathbb
F_p)\times E^r(S^*,\mathbb F_p)\longrightarrow E^r(G,\mathbb
F_p),$$ and for elements $a\in E^r_{k,l}(G,\mathbb F_p)$ and $b\in
E^r_{s,t}(S^*,\mathbb F_p)$ the following Leibniz-like rule
holds
$$
d^r(a\odot \sigma b)=(-1)^{l+t} d^r a \odot b+a\odot d^r \sigma b.
$$
\end{theorem}
\begin{proof}
It follows from Theorem \ref{Theorem_E(S^*,M)}
\end{proof}

\subsection{Unstable multiplicative spectral sequence $\tilde E(S^*, p)$}

Let $p$ be a fixed odd prime. Consider the spectral sequence
$\tilde E(X,p)$ from Example \ref{Example_E_p}. For a space $X$
and the Milnor construction $F[X]$, consider $p$-accelerated
(integral) lower central series
$$F[X]=\gamma_1[X]\supseteq
\gamma_p[X] \supseteq \gamma_{p^2}[X] \supseteq \dots.$$ Recall
that \cite{Leibowitz}, for all $m\geq 1$,
\begin{equation}\label{zero}_p\pi_{m}(\mathcal L_i(S^{n}))=0,\
\text{for}\ i\neq p^j \ \text{for some}\ j.\end{equation}
Therefore, for a sphere $S^n$, one has a natural isomorphism
$$
_p\pi_m(\mathcal
L_{p^k}(S^n))=_p\pi_m(\gamma_{p^k}F[S^n]/\gamma_{p^{k+1}}(F[S^n])),\
k\geq 1.
$$
The following follows immediately from the description of the
$E^1$-page of the lower central series spectral sequences (
\cite{6_authors} and \cite{Leibowitz})
\begin{multline*}
\pi_m(F[S^n]/\gamma_p(F[S^n])=\\ \begin{cases} \mathbb Z,\
m=n,\\
\mathbb Z\oplus \text{torsion groups of orders prime to}\ p,\
\text{if}\ n\ \text{is odd and}\ m=2n\\
0\ \text{or torsion groups of orders prime to}\ p,\
\text{otherwise}
\end{cases}
\end{multline*}
Observe that, the integral term in this description which
corresponds to the Hopf element in $\pi_{4n-1}(S^{2n})$ maps
trivially to the $\mathbb Z/p$-torsion elements in the spectral
sequence $\tilde E(S^*, p)$. This fact, of course, is well-known
in homotopy theory and can be proved in different ways. For
example, it follows from the functorial view on the spectral
sequence $\tilde E(S^*,p)$. One can extend it to the case of
wedges of spheres and consider the differentials in the spectral
sequence as natural transformations on polynomial functors from
the category of free abelian groups to all abelian groups. The
$\mathbb Z$-term in Hopf dimensions comes from the Whitehead
universal quadratic functor $\Gamma^2$, all the $\mathbb
Z/p$-terms in dimensions bellow the Hopf one come from additive
functors $-\otimes \mathbb Z/p$ and the needed statement follows
from the elementary fact that, for any odd prime $p$, there are no
non-zero natural transformations of the type $\Gamma^2\to -\otimes
\mathbb Z/p$. Hence,
$$
_p\tilde E(S^n, p)\Longrightarrow \pi_*(S^{n+1},p)
$$

Observe that, the property (\ref{zero}) can not be extended to
more complicated spaces, for example,
$$
_3\pi_5\mathcal L_{6}(S^{1}\vee S^{1})=\mathbb Z/3.
$$
Analogously with previous cases, we have
\begin{theorem}\label{simplicial_theorem1}
The Kleisli composition $\odot$ induces a well-defined map
$$\odot: \tilde E^r(S^*,p)\times \tilde E^r(S^*,p)\longrightarrow \tilde E^r(S^*,p),$$ and for elements $a\in
\tilde E^r_{k,l}(S^*,p)$ and $b\in \tilde E^r_{s,t}(S^*,p)$ the
following Leibniz-like rule holds
$$
d^r(a\odot \sigma b)=(-1)^{l+t} d^r a \odot b+a\odot d^r \sigma b.
$$
\end{theorem}
The statement of the theorem follows as before from theorem
\ref{Theorem_E(S^*,M)} and the description of the Kleisli
composition in terms of the lambda algebra.

\end{document}